\documentclass[11pt]{amsart}

\usepackage{amsmath,amsthm}
\usepackage{amsfonts}
\usepackage[nobysame]{amsrefs}
\usepackage{verbatim}
\usepackage{amssymb}
\usepackage{txfonts}
\usepackage{amscd}
\usepackage[colorlinks=true,citecolor=cyan,linkcolor=cyan,urlcolor=cyan]{hyperref}
\usepackage{mathrsfs}
\usepackage[british]{babel} % word breaking
\usepackage{soul} %striking text out
\usepackage{xcolor}

\newcommand{\set}[1]{\,\left\{#1\right\}}%set
\newcommand{\setd}[2]{\,\left\{#1\ \colon\ #2\right\}}%set with a decription usage\set{elements}

\newtheorem{theorem}{Theorem}
\newtheorem{corollary}[theorem]{Corollary}

\newcommand{\RR}{\mathbb{R}}
\newcommand{\ZZ}{\mathbb{Z}}

\newcommand{\cO}{\mathcal O}

\title{Warped cones and spectral gaps}

\subjclass[2010]{46B85 (primary), 37C85, 37A30 (secondary)}

\author{Piotr W. Nowak}
\author{Damian Sawicki}

\address{Institute of Mathematics of the Polish Academy of Sciences, Warsaw, Poland; Institute of Mathematics, University of Warsaw, Poland}

\urladdr{\url{www.mimuw.edu.pl/~pnowak/}}

\address{Institute of Mathematics of the Polish Academy of Sciences, Warsaw, Poland}
\urladdr{\url{www.impan.pl/~dsawicki/}}

\date{22 April 2016}

\thanks{Both authors were partially supported by Narodowe Centrum Nauki grant DEC-2013/10/EST1/00352. The second author was partially supported by Fundacja na rzecz Nauki Polskiej grant MISTRZ 5/2012 of Prof. Tadeusz Januszkiewicz}

\begin{document}
\begin{abstract}
We show that warped cones over actions with spectral gaps  do not embed coarsely into large classes of Banach spaces.
In particular, there exist warped cones over actions of the free group that do not embed coarsely into $L_p$-spaces and there are 
warped cones over discrete group actions that do not embed into any Banach space with non-trivial type.
\end{abstract}

\maketitle

\section{Introduction}
The goal of this paper is to construct new examples of metric spaces that do not admit coarse embeddings into various classes of Banach spaces, in particular 
into Banach spaces with 
non-trivial type. 
The only examples with such properties known previously were expanders constructed from groups with strong Banach property (T) \cites{lafforgue-duke,liao}.
The spaces we construct are warped cones over certain actions with spectral gaps. 
Warped cones were introduced by Roe in \cite{roe-warped},
where it was shown that a warped cone over an action of a countable dense subgroup $\Gamma$ of a compact Lie group $G$ on $G$
does not embed coarsely into the 
Hilbert space unless $\Gamma$ is a-T-menable.

This result was extended to more general actions of non-a-T-menable groups on metric spaces in \cite{sawicki}. 
The existence of an invariant measure admitting a positive-measure subset on which the action is free guarantees non-embeddability of the warped cone.

Our proof of non-embed\-dability is different than the ones used in \cites{roe-warped,sawicki}. In particular, it does not use negative definite 
functions or kernels, which are restricted to the setting of Hilbert spaces. This allows to bypass the use of the Gelfand--Naimark--Segal construction and apply these techniques
to general Banach spaces. It also sidesteps the requirement that the action is free.

We also obtain additional information about the distortion of the level sets of the cone. Distortion measures the best
Lipschitz constant for an embedding of a compact, usually finite, metric space into another metric space. This notion has 
many applications in computer science, where the ability to embed finite metric spaces into the Hilbert space with
small distortion gives very useful information about the computational aspects of their metrics. It is a classical result of 
Bourgain that the Euclidean distortion of any finite metric space is at most logarithmic in its cardinality, up to a universal constant. This upper bound is realized by sequences of expanders. A quantitative version of our argument
shows that discretizations of level sets of a warped over an action with a spectral gap attain the logarithmic 
upper bound. 

One of the advantages of the arguments we use is that in many situations
they allow to use the spectral gap directly, without appealing to property (T). We show for instance that there are warped cones over a-T-menable groups that do not embed coarsely into the Hilbert space, or more generally, into any $L_p$-space, $1< p<\infty$.
An illustrative example  is the warped cone $\mathcal{O}_{\mathbb{F}_n}\!\operatorname{SU}(2)$.
Bourgain and Gamburd proved \cite{bourgain-gamburd} that for
appropriately chosen subgroups $\mathbb{F}_n$ in $\operatorname{SU}(2)$ the action of $\mathbb{F}_n$ on $\operatorname{SU}(2)$ has a spectral gap and 
it follows from the result presented here that the associated warped cone is not coarsely embeddable into any $L_p$-space, $1< p<\infty$.
Another example for which the same conclusion holds is the warped cone over the action of $\operatorname{SL}_2(\ZZ)$ on the torus.

We also obtain examples of warped cones that do not embed coarsely into more general classes of spaces.  Using Lafforgue's Banach strong property (T)
\cite{lafforgue-duke} we deduce that for an ergodic action of a lattice in $\operatorname{SL}_3(\mathbb{Q}_p)$ on a probability compact metric space the 
associated warped cone does not embed coarsely into any Banach space with non-trivial type. We also discuss other examples based on strong property (T)
for classes of Banach spaces satisfying certain type and cotype conditions, proved recently  for lattices in Lie groups \cite{delasalle, delasalle-delaat,liao} and for automorphism groups
of buildings \cite{oppenheim1,oppenheim2}.

The non-embeddability results presented here provide additional evidence supporting a conjecture made in \cite{drutu-nowak}
that warped cones over actions with spectral gaps do not satisfy the coarse Baum--Connes conjecture.

\subsection*{Acknowledgements} We would like to thank the referee for suggesting several improvements.

\section{Warped cones}

Let $(Y,d)$ be a compact metric space. Up to rescaling the metric we can assume $\operatorname{diam} Y\le 1$. Consider the (truncated) Euclidean cone $\operatorname{Cone}(Y)$ over $Y$, which can be identified as a set with $Y\times [1,\infty)$. 
The metric then satisfies 
$$d_{\operatorname{Cone}}((x,t),(y,t))=td(x,y).$$
Let $\Gamma$ be a finitely generated group acting on $Y$ by homeomorphisms.
The \emph{warped cone} $\mathcal{O}_\Gamma Y$ is the set $Y\times [1,\infty)$ with the largest metric $d_\Gamma$ satisfying
the conditions
\begin{enumerate}
\item $d_\Gamma(w,w')\le d_{\operatorname{Cone}}(w,w')$,
\item $d_\Gamma(w,\gamma w)\le \vert \gamma\vert$,
\end{enumerate}
for any $w,w'\in Y\times [1,\infty)$ and $\gamma\in \Gamma$ \cite{roe-warped}.

\section{Main result}

Let $\Gamma$ be generated by a finite set $S$. Suppose that $\Gamma$ acts on a compact metric probability space $(Y,d,m)$
by measure preserving homeomorphisms.
Let $E$ be a Banach space and let $L_p(Y,m;E)$ denote the associated Bochner space, where $1<p<\infty$. The exact value of $p$ in applications is chosen depending
on the context, for instance if $E$ is some $L_q$-space then we choose $p=q$.
Denote by $\pi$ the isometric representation of $\Gamma$ on $L_p(Y,m;E)$ induced by the action,
$$\pi_\gamma f(y)=f(\gamma^{-1}y),$$
for every $\gamma\in \Gamma$ and $f\colon Y\to E$.
Let $$Mf=\int_Y f\,dm$$
denote the mean value of $f$. In the case when the action of $\Gamma$ on $Y$ is ergodic, $M$ is an equivariant projection onto 
the subspace of invariant vectors. 
We will denote the kernel of this map $L_p^0(Y,m;E)$.

Recall that the action has a spectral gap if there is $\kappa>0$ such that for any $f\in L_p^0(Y,m;E)$ we have $$\sup_{s\in S}\|f-\pi_s f\| \geq \kappa \|f\|.$$
Since the $L_p$-norm is given by integration, it is straightforward to check that if the action on $L_p(Y,m;\RR)$ has a spectral gap, then for $E=L_p(\Omega,\nu)$ the action on $L_p(Y,m;E)$ also has a spectral gap.

Any ergodic action of a Kazhdan group has a spectral gap in $L_p$, see for instance \cite[Theorem A]{bfgm} for a more general fact. The same holds for a pair $(\Gamma, H)$ with relative property (T) if the action of $H$ is ergodic.

Let $X$, $Z$ be metric spaces. A map $f\colon X\to Z$ is a coarse embedding if there exist two nondecreasing and tending to infinity functions, $\rho_-,\rho_+:[0,\infty)\to [0,\infty)$ such that 
$$\rho_-(d_X(x,x'))\le d_Z(f(x),f(x')) \le \rho_+(d_X(x,x')),$$
for all $x,x'\in X$. Equivalently, we have $\lim_n d_Z(f(x_n),f(x_n'))=\infty$ if and only if $\lim_n d_X(x_n,x_n') = \infty$ for any sequences $(x_n)$ and $(x_n')$ in $X$. We refer to \cites{nowak-yu,roe-book} for details.

\begin{theorem}\label{main thm}
Let $\Gamma$ act by measure preserving homeomorphisms on a non-atomic probability metric space $(Y,d,m)$ and let $E$ be a Banach space.
If the representation $\pi$ of $\Gamma$ on $L_p(Y,m;E)$ has a spectral gap then the warped cone $\mathcal{O}_\Gamma Y$ does not admit a coarse 
embedding into $E$.
\end{theorem}

\begin{proof} 
Let $f\colon \cO_\Gamma Y \to E$ be a coarse embedding. The restriction $f_t$ of $f$ to a section $Y\times \{t\}$ can be viewed as an element of $L_p(Y,m;E)$. 
Replacing $f_t$ by $f_t-Mf_t$ we can assume that it is an element of $L_p^0(Y,m;E)$.

Recall that for $y\in Y$ and $s\in S$, the warped distance $d_\Gamma((y,t),(sy,t))$ is at most $1$. Consequently,
\begin{align*}
 \|f_t-\pi_s f_t\|_{L_p^0(Y,m;E)}^p & = \int_Y \|f_t(y) - f_t(s^{-1}y)\|_E^p\,dm(y) \\ 
 & \leq \int_Y \rho_+(1)^p\,dm(y)\\
 & = \rho_+(1)^p.
 \end{align*}
The spectral gap guarantees that
\begin{equation}\label{equation : spectral gap}
\|f_t\|_{L_p^0(Y,m;E)} \leq \dfrac{\rho_+(1)}{\kappa}.
\end{equation}
On the other hand we have
\begin{align}\
\iint_{Y\times Y} \|f_t(x)-f_t(y)\|^p_E\, dm(x)\,dm(y) & \leq 2^{p} \int_Y \|f_t(x)\|^p_E\,dm(x) \label{equation: integral of f(x)-f(y)} \\ 
& = 2^{p} \|f_t\|_{L_p^0(Y,m;E)}^p.\nonumber
\end{align}
Thus it suffices to show that the first term expressed by a double integral is unbounded.

Since the measure is non-atomic, by the Fubini theorem the set $$N = \setd{(y,\gamma y)}{y\in Y,\, \gamma\in \Gamma}$$ is of measure $0$. By  \cite{sawicki}*{Remark 3.1} for any $x,y\in Y$ lying in different $\Gamma$-orbits the distance $d_\Gamma\left((x,t),(y,t)\right)$ goes to infinity with $t$. Since $f$ is a coarse embedding, also $\|f_t(x) - f_t(y)\|_E^p$ tends to infinity. 
Thus, the integrand of $\iint_{Y\times Y} \|f_t(x)-f_t(y)\|^p_E\, dm(x)\,dm(y)$ tends to infinity almost everywhere, so the integral itself goes to infinity, yielding the desired contradiction.
\end{proof}

Assume now that $Y$ is a compact metric space admitting a bi-Lipschitz embedding into a Banach space $E$ (this is automatically satisfied when $Y$ is finite).
Given such a bi-Lipschitz embedding $f\colon Y\to E$ define 
$$c_{E,f}(Y)=\sup_{x,y\in Y}\dfrac{\Vert f(x)-f(y)\Vert}{d(x,y)}\cdot \sup_{x,y\in Y} \dfrac{d(x,y)}{\Vert f(x)-f(y)\Vert}.$$
The $E$-distortion is then defined to be
$$c_E(Y)=\inf\setd{ c_{E,f}(Y) }{ f\colon Y\to E \text{ is bi-Lipschitz}}.$$
In particular, the distortion $c_E(Y)$ is infinite when $Y$ does not admit a bi-Lipschitz embedding into $E$.
%In general, we adopt the convention that  $c_{E,f}(Y)$ and $c_E(Y)$ are allowed to be infinite.
For finite metric spaces the Euclidean distortion $c_2=c_{\ell_2}$  has been studied extensively and has many applications in the  
geometry of Banach spaces and theoretical computer science. A classical result of Bourgain \cite{bourgain} states that there exists a universal constant $C\ge 0$ such that
any finite metric space $Y$ satisfies $c_2(Y)\le C \log |Y|$. Bourgain also showed that this bound is tight and is achieved by expanders. Here we show that 
level sets of a warped cone also achieve this worst bound.

The additional assumptions in the next corollary are clearly satisfied for manifolds with a measure given by a volume form and smooth $\Gamma$-actions.
\begin{corollary}\label{C:distortion} 
%Let $Y$ admit a bi-Lipschitz embedding into $E$.
If $\sup_y m(B_d(y,r)) = O(r^k)$ for some $k>0$, and $\Gamma$ acts on $Y$ by Lipschitz homeomorphisms, then $Y\times \{t\} \subseteq \mathcal{O}_\Gamma Y$ has $E$-distortion of order at least $\log t$.
\end{corollary}
\begin{proof}
By \cite{drutu-nowak}*{Lemma 7.2}, for every $R>0$ there exists $K>0$ such that the ball $B(w,R)$ in the warped cone satisfies
$$B(w,R)\subseteq \bigcup_{\vert \gamma\vert\le R} B_{\operatorname{Cone}}(\gamma w, K).$$
In fact, $K$ can be chosen as $R\cdot L^R$, where $L$ is the Lipschitz constant for the action of generators.

Hence,
\begin{align*}
m\left(B((y,t),R)\cap Y\times \{t\}\right) & \leq |B(1_\Gamma,R)| \cdot \sup_{x\in Y} m\left(B_d\left(x,\frac{R\cdot L^R}{t}\right)\right)\\
& \lesssim \vert S\vert^R \cdot \frac{R^k\cdot L^{kR}}{t^k}.
\end{align*}
If we put $R=c\log t$ for $c$ small enough, the right hand side is bounded by, say, one half (compare \cite{drutu-nowak}*{Lemma 7.3}). Denote 
$$D_t(R)=\setd{(x,y)\in Y\times Y}{d_\Gamma\left((x,t),(y,t)\right)\le R}.$$
From the Fubini theorem it follows that $m\times m(D_t(R))\le 1/2.$

Let $f_t\colon Y\times\set{t} \to E$ be a bi-Lipschitz embedding and let $L$ be its optimal Lipschitz constant. In particular, we have $\|f_t(x) - f_t(y)\| \leq L$ for $d_\Gamma((x,t),(y,t)) \leq 1$. Then, combining
\eqref{equation: integral of f(x)-f(y)} and
\eqref{equation : spectral gap}, we have
\begin{align*}
\iint_{Y\times Y\setminus D_t(K)} \|f_t(x)-f_t(y)\|^p_E \,dm(x)\,dm(y)&\leq \iint_{Y\times Y} \|f_t(x)-f_t(y)\|^p_E \,dm(x)\,dm(y)\\
&\leq \frac{2^{p}\cdot L^p}{\kappa^p}.
\end{align*}
It follows that there is a subset of positive measure of $Y\times Y\setminus D_t(K)$ on which $\|f_t(x)-f_t(y)\|_E\leq\frac{4L}{\kappa}$. Hence, the distortion is bounded below by
\begin{equation*}
L\cdot\frac{R}{\frac{4L}{\kappa}} = L\cdot\frac{c\cdot \log t}{\frac{4L}{\kappa}} = \frac{c\cdot\kappa}{4} \log t.\qedhere
\end{equation*}
\end{proof}

Recall that for $C>0$ a $C$-net in a metric space $(Y,d)$ is a subset $N\subseteq Y$ such that for any $y\in Y$ there exists $x\in N$ such that $d(x,y)\le C$.
If $Y^m$ is a manifold then there is a $1$-net $Z_t$ of cardinality of order $t^m$ in $Y\times\{t\}\subset \operatorname{Cone}(Y)$ and consequently in $Y\times\{t\}\subset \mathcal O_\Gamma Y$. One can check  that  the distortion of $Y\times\{t\}\subset \mathcal O_\Gamma Y$ and $Z_t\subset \mathcal O_\Gamma Y$ are of the same order. Thus, in the case of Corollary \ref{C:distortion}, we obtain the following.
\begin{corollary}
The distortion  $c_E(Z_t)$ is at least of order $\log |Z_t|$.

 In the case $E=\ell_2$ the Euclidean distortion $c_2(Z_t)$ realizes the upper bound.
\end{corollary}

\section{Examples}
There are plenty of examples to which the above theorem applies, giving rise to warped cones non-embeddable into 
various classes of Banach spaces.

\subsection*{\texorpdfstring{Warped cones non-embeddable into $\boldsymbol{L_p}$-spaces}{Warped cones non-embeddable into L\_p-spaces}}
Many examples of group actions with spectral gaps are known. For instance, it is shown in \cite{bourgain-gamburd} that certain finite subsets of elements in 
$\operatorname{SU}(2)$ generate a free subgroup, whose action on $\operatorname{SU}(2)$ has a spectral gap in $L_2(\operatorname{SU}(2))$. 
Another example of an action with a spectral gap is the action of $\operatorname{SL}_2(\ZZ)$ on the torus $\mathbb{T}^2=\RR^2/\ZZ^2$.

\begin{corollary} The warped cones $\mathcal{O}_{\mathbb{F}_n}\!\operatorname{SU}(2)$ and $\mathcal{O}_{\operatorname{SL}_2(\mathbb Z)}\mathbb{T}^2$ do not embed coarsely into the Hilbert space.
\end{corollary}

It is worth pointing out that $\operatorname{SU}(2)$ does not contain finitely generated subgroups with property (T). Furthermore, both $\mathbb F_n$ and $\operatorname{SL}_2(\mathbb Z)$ are a-T-menable.

The aforementioned result about $\operatorname{SU}(2)$ was generalized to $\operatorname{SU}(d)$ in \cite{bourgain-gamburd-d}. More generally, a recent result \cite{benoist-desaxce} shows that many discrete subgroups of compact simple Lie groups have a spectral gap for the action of the subgroup on
the ambient compact Lie group.

A spectral gap for the representation of $\Gamma$ on $L_2(Y,m)$ induced by an action of $\Gamma$ on $Y$ 
implies that the action of $\Gamma$ on $L_p(Y,m)$ has a spectral gap for any $1< p<\infty$.
Indeed, consider a Markov operator $A^\mu_\pi$ associated to an admissible probability measure $\mu$ and the representation induced by the action of $\Gamma$ on $Y$.
The spectral gap for the action on $L_p(Y,m)$ is equivalent to the fact that $\Vert A_\pi^\mu\Vert <1$ on $L_p^0(Y,m)$, see \cite{drutu-nowak}. However, under the assumption, 
this last condition
holds by interpolation  for all $p$ in $(1,2)$ and $(2,\infty)$.

\begin{corollary} The warped cones $\mathcal{O}_{\operatorname{SL}_2(\mathbb Z)}\mathbb{T}^2$ and $\mathcal{O}_{\Gamma}G$ do not embed coarsely into any $L_p$-space for any $1\le p<\infty$, provided that $\Gamma$ is an appropriate finitely generated subgroup of a Lie group $G$ as in \cites{bourgain-gamburd,bourgain-gamburd-d,benoist-desaxce}.
\end{corollary}

%It is worth pointing out that $\operatorname{SU}(2)$ does not contain finitely generated subgroups with property (T). Furthermore, both $\mathbb F_2$ and $\operatorname{SL}_2(\mathbb Z)$ are a-T-menable.

%More generally, the above argument gives coarse non-embeddability of warped cones over certain dense subgroups of any compact simple Lie group \cite{benoist-desaxce}.

\subsection*{Warped cones non-embeddable into certain classes of Banach spaces}
Consider all spaces with type $p$ and cotype $q$ satisfying 
\begin{equation}\label{equation: type-cotype condition}
\dfrac{1}{p}-\dfrac{1}{q}\le \dfrac{1}{r}.
\end{equation}
In particular, in \cite{delasalle} de la Salle showed that $\operatorname{SL}_3(\mathbb{R})$ 
has Lafforgue's strong property (T) \cite{lafforgue-duke} for all spaces satisfying  \eqref{equation: type-cotype condition} with $r=4$.
Subsequently, in \cite{delasalle-delaat} jointly with de Laat they showed that 
connected higher-rank simple Lie groups have strong property (T) with respect to the class of spaces satisfying 
\eqref{equation: type-cotype condition} with $r=10$.
Oppenheim \cite{oppenheim1} showed recently that Steinberg groups have spectral gaps for isometric representations on
 spaces in the class  satisfying \eqref{equation: type-cotype condition} with $r=4$.
He proved a similar result for automorphism groups of thick buildings and Banach spaces in the class of spaces satisfying \eqref{equation: type-cotype condition} with $r=20$ \cite{oppenheim2}.

Since $L_2(Y,m; E)$ has type and cotype equal, respectively, to the type and cotype of $E$ (see e.g. \cite[Theorem 11.12]{jarchow-etal}), in all of the above cases the warped cone associated to any non-atomic ergodic action of $G$ does not embed coarsely into any Banach space in the corresponding class listed above.

\subsection*{Warped cones non-embeddable into spaces with non-trivial type}

Consider now a finitely generated group $\Gamma$ with strong Banach property (T) \cite{lafforgue-duke}, that is, property (T) with respect to all Banach spaces of non-trivial type. Examples of such groups include
lattices in $\operatorname{SL}_3(\mathbb{Q}_p)$, as shown by Lafforgue \cite{lafforgue-duke}. This class was extended 
in \cite{liao} to 
 connected almost $F$-simple algebraic groups whose $F$-split rank is at least 2, where $F$ denotes a non-Archimedean local field. 
Let $Y$ be any non-atomic probability space on which $\Gamma$ acts ergodically.
\begin{corollary}
The warped cone $\mathcal{O}_\Gamma Y$ for $\Gamma$ and $Y$ as above does not embed coarsely into any Banach space with non-trivial type.
\end{corollary}

\end{document}